\documentclass{amsart}
\usepackage[margin=1in]{geometry}
\usepackage{amsmath}
\usepackage{amssymb}
\usepackage{amsthm}
\usepackage{mathrsfs}
\usepackage{enumerate}
\usepackage{tikz}
\usetikzlibrary{decorations.pathreplacing}
\newtheorem{thm}{Theorem}[section]
\newtheorem{lem}[thm]{Lemma}

\theoremstyle{definition}
\newtheorem{dfn}[thm]{Definition}

\newtheorem*{prb*}{Problem}

\newtheorem{rmk}[thm]{Remark}

\usepackage[indent]{parskip}

\numberwithin{equation}{section}

\usepackage{hyperref}

\providecommand{\al}{\alpha}

\providecommand{\de}{\delta}

\providecommand{\e}{\varepsilon}
\providecommand{\f}{\varphi}

\providecommand{\N}{\mathbb{N}}

\providecommand{\R}{\mathbb{R}}

\providecommand{\E}{\mathbb{E}}

\providecommand{\cD}{\mathcal{D}}

\providecommand{\cQ}{\mathcal{Q}}
\providecommand{\cR}{\mathcal{R}}
\providecommand{\cS}{\mathcal{S}}

\DeclareMathOperator{\Fav}{Fav}
\DeclareMathOperator{\FavP}{\Fav_{\Phi}}

\title{Favard length and generalized projections}

\author{Izabella  {\L}aba }
\address{Mathematics Department, The University of British Columbia}
\email{ilaba@math.ubc.ca }

\author{Alex McDonald}
\address{Mathematics Department, Kennesaw State University, Marietta, GA}
\email{amcdon79@kennesaw.edu}

\author{Krystal Taylor}
\address{Mathematics Department, The Ohio State University, Columbus, OH}
\email{taylor.2952@osu.edu}

\begin{document}
\begin{abstract}
We investigate generalized Favard lengths associated to smooth families
of nonlinear projections. Under suitable regularity and
transversality assumptions, we prove that generalized projections
are locally comparable to orthogonal projections on sufficiently
small scales. This yields a comparison principle that transfers
quantitative upper bounds for classical Favard length to broad
classes of nonlinear projection families.

As a consequence, known upper bounds for the Favard length of
purely unrectifiable self-similar $1$-sets 
yield corresponding upper bounds for their
generalized Favard lengths. We also prove that the union of
circles with centers in a purely unrectifiable self-similar
$1$-set has Lebesgue measure zero whenever the radii vary
sufficiently slowly. More generally, the same method yields
measure estimates for unions of curves arising from suitable
level-set families.
\end{abstract}

\maketitle


\section{Introduction}
This paper contributes to the study of geometric properties that remain invariant under nonlinear projection methods, as well as to Kakeya-type problems for curved objects. We begin with the connection between classic orthogonal projections and rectifiability. 
\subsection{Orthogonal projections $\&$ rectifiability}
Orthogonal projections are fundamental tools in geometric measure theory and fractal geometry. They provide lower-dimensional views of a set while preserving geometric information and play a central role in the study of dimension, measure, and rectifiability. A cornerstone of the subject is Marstrand's projection theorem, which exhibits a threshold at dimension $1$: if a planar set has Hausdorff dimension at most $1$, then its dimension is preserved under projection onto lines in almost every direction; whereas if its Hausdorff dimension exceeds $1$, then its projection has positive Lebesgue measure in almost every direction \cite{Mar54,Mattila2015}. Projection theorems of this type reveal deep connections between geometric measure theory, harmonic analysis, and incidence geometry.

Closely related to the critical dimension $1$ in Marstrand's theorem is the \emph{Favard length}, which measures the average length of orthogonal projections and plays a central role in the study of rectifiability.
For a Borel set
$E \subset \mathbb{R}^2$, it is defined by
\[
\Fav(E)
=
\int_0^\pi |\pi_\theta(E)|\, d\theta,
\]
where $\pi_\theta$ denotes orthogonal projection onto the line making angle $\theta$ with the $x$-axis. Equivalently, up to normalization, $\Fav(E)$ is the probability that a long needle dropped at random in the plane intersects the set $E$. This geometric interpretation traces back to Buffon's famous needle problem and provides a natural connection between projections and geometric probability \cite{Mat95}.

Favard length detects rectifiability and plays a role in the quantitative study of rectifiability. A classical theorem of Besicovitch asserts that a set of positive and finite length is purely unrectifiable if and only if its Favard length vanishes \cite{Bes39,Mat95}. Quantitative versions of this phenomenon have attracted significant attention over the last several decades. Given a fractal set $E$, one studies the decay of the Favard length of finite approximations $E_n$ as $n\to\infty$. The prototypical example is the four-corner Cantor set $K$, whose $n$-th approximation $K_n$ consists of $4^n$ squares of side length $4^{-n}$. Besicovitch's theorem implies that $\Fav(K_n)\to 0$, but determining the precise rate of decay remains a major open problem.

Substantial progress has nevertheless been achieved. Peres and Solomyak established the first explicit quantitative upper bounds \cite{PS02}. Later, Nazarov, Peres, and Volberg obtained the first power-law estimate
for the 4-corner set,
representing a breakthrough in the subject \cite{NPV10}. Since then, further developments have been obtained for a variety of self-similar constructions, including rational product Cantor sets, Sierpi\'nski-type constructions, and other classes of planar self-similar sets
\cite{BLV14,BoV12,LM22,LZ10}. These works collectively reveal a meaningful interplay between harmonic analysis, combinatorial geometry, self-similarity, and arithmetic structure.

\subsection{Nonlinear projections $\&$ generalized Favard length} 

The Favard length problem has also inspired a number of
nonlinear variants in which orthogonal projections are
replaced by more general projection families. Examples
include radial projections, visibility-type quantities,
projection operators associated to curves, and more general
transversal families of nonlinear projections
\cite{BongersTaylor2023,Orp19,PeresSchlag}.
Such objects arise naturally in geometric measure theory
and often encode geometric information.

A notable example is the Favard curve problem studied by
Cladek, Davey, and Taylor \cite{CDT21}, where orthogonal
projections are replaced by projection families associated
to translates of a fixed curve. The authors established
upper and lower bounds on the rate of decay of the Favard
curve length of finite approximations to the four-corner
Cantor set. More generally, recent work on transversal
families of nonlinear projections suggests that many
projection-theoretic phenomena persist beyond the linear
setting \cite{BongersTaylor2023}.

In the current article, we investigate generalized Favard lengths associated to smooth families of nonlinear projections. Let
\[
\Phi_\alpha: U\to\R,\ \alpha\in I,
\]
be a smooth family of projection maps, and define
\[
\FavP(E)
=
\int_I |\Phi_\alpha(E)|\,d\alpha.
\]
The classical Favard length is recovered when $\Phi_\alpha$ are orthogonal projections onto lines. Our goal is to understand how generalized Favard lengths compare to the classical Favard length and to what extent estimates from the linear setting persist for nonlinear projection families. Further, we investigate applications to large unions of circles and other curved objects parameterized by  purely unrectifiable sets.

\subsection{Local comparison - our main result $\&$ key difficulties}

The main contribution of this paper is a local comparison principle
showing that, under suitable regularity and transversality assumptions,
nonlinear projection families are quantitatively comparable to
orthogonal projections on sufficiently small scales.
The key observation underlying this comparison principle is that smooth
projection families appear linear at sufficiently small scales.

While our arguments are inspired in part by ideas arising in the study of Favard curve length \cite{CDT21}, the setting considered here is considerably more general. In \cite{CDT21}, the fibers are translates of a fixed curve, whereas in the present setting they may vary with both the projection parameter and the variable location. Consequently, one no longer has a global geometric description of the fibers, and additional work is needed to obtain the uniform geometric control required for comparison with orthogonal projections.  Our approach is to show that, on sufficiently small scales, the fibers may be efficiently contained in thin rectangles whose long sides are approximately tangent to the fibers. This provides the uniform geometric control needed to compare nonlinear and orthogonal projections.

The primary strength of this perspective is its flexibility. Once the local comparison principle has been established, quantitative Favard-length estimates from the classical projection setting can be transferred to a broad class of nonlinear projection families. This allows us to obtain new applications for generalized projections, including results for self-similar sets, random constructions, and unions of curves.

\subsection{Applications of our main result}
As an application of our comparison principle, we show that quantitative upper bounds for the classical Favard length yield corresponding upper bounds for generalized Favard lengths. 
In particular, known estimates for the four-corner Cantor set and related self-similar constructions, as well as upper bounds for related probabilistic constructions, transfer directly to the nonlinear setting considered here. 

 A second application concerns unions of curves arising from level-set
families. We prove a general theorem (Theorem \ref{thm-Gfunction}) showing that the local comparison
principle can be used to obtain measure estimates for unions of curves
whose geometry is encoded by a suitable projection family. As concrete
examples, we study unions of circles with slowly varying radii and
families of ellipses with slowly varying parameters. In particular, Theorem \ref{thm:unioncircles}  shows that a union of circles with centers in a purely unrectifiable
self-similar set has Lebesgue measure zero whenever the radii vary
sufficiently slowly.

\subsection{Notation and Preliminaries}
\label{notation}

By ``direction'' we mean an element of the projective line $\mathbb{P}^1$, i.e., an element of the circle with antipodal points identified.  The direction of a non-zero vector $v\in\R^2$ is the equivalence class of $\frac{v}{|v|}$ under this identification, and the direction of a line segment $L$ is the direction of $x-y$ for any distinct $x,y\in L$.  If $R$ is a rectangle which is not a square, the direction of $R$ is the direction of the longer sides.  We will also identify directions with elements of $\R/\pi\mathbb{Z}$ in the usual way, so that $\theta\in\R$ is identified with the direction $(\cos\theta,\sin\theta)$. We will also write $\theta^\perp=\theta-\pi/2$, corresponding to the direction $(\sin\theta,-\cos\theta)$.

For $E\subset \R^d$, we let $B_\e(E)$ denote the $\e$-neighborhood of the set $E$.  In particular, when $E=\{x\}$ is a single point, $B_\e(x)$ is the ball of radius $\e$ centered at $x$.

If $A$ and $B$ are two quantities, we write $A\lesssim B$ if $A\leq CB$ for some constant $C>0$ which may depend on $\Phi$, $U$, $V$, and $I$ (defined in Theorem \ref{thm: mainlocal}) but are independent of $x$, $\alpha$, small parameters such as $\delta$, and the choices of various squares and rectangles appearing in our results. We also write $A\approx B$ if we have both $A\lesssim B$ and $B\lesssim A$.

\section{Main Results}


We begin by defining the nonlinear analogue of Favard length that will be studied throughout the paper.

\begin{dfn}[Generalized Favard length]
Let $I\subset \R$ be an open bounded interval, and let $U\subset \R^2$ be an open set.  For $\Phi:I\times U\to \R$ and $\al\in I,x\in U$, let
\[
\Phi_\al(x)=\Phi(\al,x).
\]
For $E\subset U$, define the \textbf{generalized Favard length} of $E$ with respect to $\Phi$ to be
\[
\FavP(E)=\int_I |\Phi_\al(E)|d\al.
\]
\end{dfn}
The classical Favard length $\Fav(E)$ is a special case, with
\[
\Fav(E)=\int_0^\pi |\pi_\alpha(E)|d\alpha,
\]
where $\pi_\alpha(x) =x_1\cos{\alpha}+ x_2\sin{\alpha}$ for $x=(x_1,x_2)\in\R^2$. Thus $\FavP(E)=\Fav(E)$ when $\Phi_\alpha=\pi_\alpha$ and $I=[0,\pi]$.

\subsection{Local comparison of projection families}
The main goal of the paper is to compare generalized Favard lengths with the classical Favard length. Our first results, building on
the proof methods introduced in Proposition 6 and Lemma 7 in \cite{CDT21}, offer a quantitative local comparison. We start with an upper bound.

\begin{thm}[Local comparison of linear and generalized Favard lengths]
\label{thm: mainlocal}
Let $I\subset \R$ be an open bounded interval, and $U\subset \R^2$ an open set.  Let $\Phi\in C^2(I\times U)$ be a real-valued function such that the following conditions hold uniformly over $\al\in I,x\in U$:
\begin{equation}
\label{gradient}
|\nabla \Phi_\al(x)|\approx 1
\end{equation}
\begin{equation}
\label{seconddv}
|\partial_i\partial_j\Phi_\al(x)|\lesssim 1 \textup{ for all } i,j=1,2.
\end{equation}
and
\begin{equation}
\label{jacobian}
\left|
\det
\begin{pmatrix}
\partial_1 \Phi_\al(x) & \partial_\al\partial_1 \Phi_\al(x) \\
\partial_2 \Phi_\al(x) & \partial_\al\partial_2 \Phi_\al(x)
\end{pmatrix}
\right|\approx 1.
\end{equation}
Let $V\subset U$ be a fixed compact set.
Then, there exists $\de_0>0$ (depending on $\Phi,V,I$) such that the following holds.  If $Q\subset V$ is a compact square of side length $\de<\de_0$ and $E\subset Q$ is a union of finitely many compact squares of side length $\de^2$, then
\begin{equation}\label{eq-compare-squares}
\FavP(E)\lesssim \Fav(E)
\end{equation}
with the implicit constant independent of $Q$, $E$, and $\delta$.
\end{thm}
\bigskip

Theorem \ref{thm: mainlocal} is based on a local similarity between linear and nonlinear projections, as follows. 
Let $\Phi_\alpha:U\to\R$, where $U\subset\R^2$ is an open set, be a $C^2$ mapping such that $\nabla\Phi_\alpha\neq 0$ on $U$. Define
$\theta_\alpha(x)$ via
\begin{equation}\label{defn: theta alpha}
\left( \cos{\theta_\alpha(x)}, \sin{\theta_\alpha(x)}\right)
=\frac{\nabla\Phi_\alpha(x)}{|\nabla \Phi_\alpha(x)|}.
\end{equation}

Geometrically, $\theta_\alpha(x)$ is the direction parallel to $\nabla \Phi_\al(x)$
(or, equivalently, perpendicular to the level curve of $\Phi_\alpha$ through $x$).
Then, near a fixed point $x=a$,
$\Phi_\alpha(x)$ is well approximated by the linear function
$$
\Phi_\alpha(a)+|\nabla \Phi_\alpha(a)|\Big( \cos(\theta_\alpha(a))(x_1-a_1)
+\sin(\theta_\alpha(a))(x_2-a_2)
\Big).
$$
This will allow us to exploit the local similarity between linear projections and the generalized projections defined by $\Phi_\alpha$.

In general, we do not expect the upper bound \eqref{eq-compare-squares} to be reversible. While the local linearization works for a fixed parameter $\alpha$, it need not be the case that every direction appearing in the definition of $\Fav(E)$ can be realized by the nonlinear projection family. Nevertheless, we have a localized lower bound involving only those angles that can be so represented.

\begin{thm}[Lower bound for the local comparison]
\label{thm: mainlocal-lower}
Let $I,U, V,\Phi, Q, E$ satisfy the assumptions of Theorem \ref{thm: mainlocal}. Assume in addition that $\Phi$ satisfies
\begin{equation}
\label{components}
(\forall y\in \R)\  Q\cap \Phi_\al^{-1}(y) \textup{ has a (uniformly) bounded number of connected components,}
\end{equation}
and that there is an interval $\Theta\subset [0,2\pi]$ such that for some point $x_0\in Q$,
\begin{equation}\label{angle-range}
\Theta\subset\{\theta_\alpha(x_0):\ \alpha\in I\}.
\end{equation}
Then we have the lower bound
\[
\FavP(E)\gtrsim  \int_\Theta \left| \pi_\theta(E)\right| d\theta.
\]

\end{thm}

Together, Theorems \ref{thm: mainlocal} and
\ref{thm: mainlocal-lower} show that generalized Favard lengths
are locally comparable to classical Favard lengths 
on the range of directions realized by the projection family.

\subsection{Self-similar sets}\label{subsec-selfsim}

As a first application, we transfer classical Favard-length estimates to generalized projection families.

Let $L\in\mathbb{N}$ with $L\geq 3$, and let $b_1,\dots,b_L\in\R^2$. Let $E\subset\R^2$ be the unique non-empty compact set satisfying
$$
E=\bigcup_{i=1}^L \Lambda_i(E), \hbox{ where }\Lambda_i(x)=L^{-1}x+b_i.
$$
We assume that the Strong Separation Condition holds, so that $\Lambda_i(E)\cap \Lambda_j(E)=\emptyset$ for $i\neq j$. This implies that $E$ has both the similarity dimension and the Hausdorff dimension equal to 1 \cite{Hutch1981}. We will further assume that the points $b_i$ are not all collinear, so that $E$ is purely unrectifiable.

Define the finite iterations of $E$ inductively by
$$
E_0=[0,1]^2,\ E_n=\bigcup_{i=1}^L \Lambda_i(E_{n-1})\hbox{ for }n\in\mathbb{N}.
$$
Intuitively, we may think of $E_n$ as $L^{-n}$-neighbourhoods of $E$; this is not actually true, but $E_n$ may be covered by $O(1)$ such neighbourhoods with the big-$O$ constant independent of $n$, and vice versa. 

By Besicovitch's classic theorem \cite{Bes39}, we have $\Fav(E)=0$, hence $\Fav(E_n)\to 0$ as $n\to\infty$. For self-similar sets defined as above, we have the quantitative bound
\begin{equation}\label{eq:subexp}
cn^{-1} \leq \Fav(E_n)\leq C\exp(-c_0\sqrt{\log n}).
\end{equation}
The lower bound is due to Mattila \cite{Mat90}, and the upper bound was proved by Bond and Volberg \cite{BoV12}. 
Better bounds (both upper and lower) are known in a number of special cases, see \cite{BaV10, BLV14, BoV10, CSS26, LM22, LZ10, NPV10}. For example, for the four-corner set $K_n$ we have the estimate
\begin{equation}\label{NPV-bound}
\Fav(K_n)\lesssim_\e n^{-\frac{1}{6}+\e}
\end{equation}
for all $\e>0$ \cite{NPV10}; see also \cite{Bond2011}. 


\begin{thm}\label{thm:comparison} (Global comparison).
Let $E_n$ be the $n$-th iteration of a purely unrectifiable self-similar 1-set as described above.
With $\Phi$ satisfying the assumptions of Theorem \ref{thm: mainlocal}, we have
\begin{equation}\label{eq: Fav compare}
\frac{1}{n}\lesssim
\FavP(E_{n}) \lesssim \Fav(E_{\lfloor n/2\rfloor}).
\end{equation}
Thus, asymptotic upper bounds on the Favard length of $E_n$ (such as the upper bounds in (\ref{eq:subexp}) and (\ref{NPV-bound})) also hold for the generalized Favard length $\FavP(E_n)$.
\end{thm}

The upper bound in Theorem \ref{thm:comparison} is derived from the local comparison in Theorem \ref{thm: mainlocal} above, while the lower bound follows from the main result of the third author's joint work with Bongers, namely \cite[Theorem 1.5]{BongersTaylor2023}. Further details are provided in the proof.

\begin{rmk}
The appearance of $\lfloor n/2\rfloor$ in Theorem \ref{thm:comparison}
reflects the way the local comparison estimate is combined with the
self-similar structure of the approximating sets. Roughly speaking, the
local comparison theorem is applied on pieces of diameter $\delta$, while
the approximation occur at scale $\delta^2$.
Subadditivity then relates the generalized Favard length at level $n$ to
the classical Favard length at a coarser level $n/2$.
\end{rmk}

\subsection{Probabilistic Constructions}\label{subsec-prob_constructions}
Beyond self-similar sets, the bounds in \ref{eq: Fav compare} extend to certain probabilistic constructions, such as those in \cite{CSS26}.

Fix $L\geq 2$, divide the unit square into an $L^n \times L^n$ grid, 
and let $\mathcal{D}_n$ be the resulting set of squares. 

Create a planar $1$-dimensional random fractal $S$ as follows: 
\begin{enumerate}
    \item Subdivide $S_0=[0,1]^2$ into an $L\times L$ grid of squares, and choose $L$ of these squares at random according to the probability distribution $\chi$ on $\binom{\mathcal{D}_1}{ L}$. Let $S_1 \subset S_0$ denote the union of the chosen squares. 
    \item Repeat the step above independently in each of the $L$ squares in $S_1$, again using the distribution $\chi$, and get $S_2\subset S_1$, where $S_2$ is made up of $L^2$ squares of $\mathcal{D}_2$. 
    \item Repeat the process indefinitely to obtain a nested sequence of random sets $S_0\supset S_1 \supset S_2 \supset \cdots$ and a limiting set $S= \bigcap_n S_n$. 
\end{enumerate}
The probability distribution $\chi$ is said to be \textbf{vertically degenerate} if it selects one square in each column almost surely, and \textbf{horizontally degenerate} if it selects one square in each row almost surely.  It is said to be \textbf{degenerate} if it is either horizontally or vertically degenerate, and is \textbf{non-degenerate} otherwise.  Finally, $\chi$ is \textbf{uniform} if, for any square $Q$ in the initial $L\times L$ grid, the probability of choosing $Q$ for $S_1$ is $1/L$.

Recent work of Chang, Shmerkin, and Suomala \cite{CSS26} establishes the bounds
\begin{equation}
\label{eq: CSS nd}
\E[\Fav(S_n)]\approx \frac{1}{n}
\end{equation}
when $\chi$ is uniform and non-degenerate, and
\begin{equation}
\label{eq: CSS d}
\E[\Fav(S_n)]\approx \frac{\log(n)}{n}
\end{equation}
when $\chi$ is uniform and degenerate.

\begin{thm}
\label{thm: CSS generalization}
Let $\Phi$ satisfy the hypotheses of Theorem \ref{thm: mainlocal} with $[0,1]^2\subset U$, and let $S_n$ be as constructed as above with respect to a uniform $\chi$.  We have
\[
\E[\FavP(S_n)]\lesssim \E[\Fav(S_{\lfloor n/2\rfloor})],
\]
with constant depending only on $\Phi,I,$ and $L$.  In particular, the upper bound in (\ref{eq: CSS nd}) (respectively, (\ref{eq: CSS d})) holds with $\FavP(E)$ in place of $\Fav(E)$ when $\chi$ is non-degenerate (respectively, degenerate).
\end{thm}

\subsection{Unions of circles and level curves}
Our second application concerns unions of circles and, more generally,
unions of level curves. Questions about unions of circles arise naturally
in geometric measure theory and harmonic analysis, see \cite{Mitsis1999, Oberlin2006, Oberlin2007, Wolff1997, Wolff2000}. Here, we focus on the following question. 
For $x\in \R^2$ and $r>0$, let $C(x,r)$ denote the circle centered at $x$ of radius $r=r(x)$, 
$$C(x,r) = \{y\in \R^2: |x-y| = r\}.
$$  
Given a compact set $E\subset\R^2$ (the set of centers) and a function $r:E\to (0,\infty)$, what can we say about the size of the set 
$$\mathcal{E}:= \bigcup_{x\in E} C(x,r(x))?$$
Suppose that $E$ has Hausdorff dimension $\alpha\in[0,2]$. Wolff \cite{Wolff2000} proved that if $\alpha>1$, then $\mathcal{E}$ has positive 2-dimensional Lebesgue measure. For $\alpha<1$, Oberlin \cite{Oberlin2007} proved that $\mathcal{E}$ has Hausdorff dimension at least $(1+\alpha)$. This is consistent with the natural heuristic that a one-parameter family of circles contributes at least one additional dimension beyond that of the parameter set. 

The critical case $\alpha=1$ is more subtle. In this case, the result of \cite{Oberlin2007} implies that $\mathcal{E}$ has Hausdorff dimension 2, but not necessarily positive Lebesgue measure. Indeed, 
Talagrand \cite{Talagrand} constructed a set of 2-dimensional Lebesgue measure zero containing a circle centered at every point of a line.
For circles of fixed radius,
Simon and Taylor \cite{STdim} showed that rectifiability rather than dimension alone determines whether the union has positive Lebesgue measure. More precisely, for $1$-sets $E$, the union of circles $\mathcal{E}$ with fixed $r(x)=r_0$ has positive Lebesgue measure if and only if $E$ is not purely unrectifiable.

The fixed-radius result of Simon and Taylor is closely related to
projection theory: the circle condition
$|x-y|=r_0$ may be viewed as the level-set equation
$G(x,y)=|x-y|-r_0=0$. More generally, if the radius varies with the
center $x$, then the circles arise as level curves of
$G(x,y)=|x-y|-r(x)$, placing the problem within the broader class of
level-curve families considered in this paper.

One of the motivations for the framework developed here is that it
allows such families of level curves to be studied through a common
projection mechanism. The approach is quite different from that of
Simon-Taylor \cite{STdim}. Rather than constructing and analyzing a
specific family of nonlinear projections associated with circles, we
develop a general local comparison principle (Theorems
\ref{thm: mainlocal} and \ref{thm: mainlocal-lower}) and use it to
transfer Favard-length estimates from the classical linear setting
(Theorem \ref{thm:comparison}). Theorem
\ref{thm-Gfunction} then converts these projection-theoretic estimates
into measure estimates for unions of level curves. Theorem
\ref{thm:unioncircles} is obtained as a special case of this more
general framework.

Our next theorem shows that the measure-zero direction of this result persists when the radius is allowed to vary slowly.

\begin{thm}[Large unions of circles with centers in a purely unrectifiable set have measure zero]\label{thm:unioncircles}
Let $0<a<b<\infty$ and $0\leq c<1$ be fixed real numbers. Let also
$r: \R^2\to [a,b]$ be a $C^2$ function such that
\begin{equation}\label{grad-r-small-intro}
|\nabla r|\leq c.
\end{equation}
    Let $E$ be a purely unrectifiable self-similar 1-set as defined above.
    Then $$\bigcup_{x\in E} C(x,r(x))$$ 
    has 2-dimensional Lebesgue measure zero. Furthermore, if
 $E_n\subset\R^2$ is the $n$-th approximation of $E$, then
 $$\left|\bigcup_{x\in E_n} C(x,r(x))\right|\lesssim \Fav(E_{\lfloor n/2\rfloor} ). $$
\end{thm}

Theorem \ref{thm:unioncircles} follows from Theorem \ref{thm-Gfunction}, a more general result on unions of curves defined as level curves of a function satisfying appropriate regularity and curvature conditions. Since the statement of Theorem \ref{thm-Gfunction} is longer and more complicated technically, we defer it to Section \ref{sec-Gfunction}, along with the proofs of both theorems.

\subsection{Organization of the paper}
Section \ref{notation} introduces notation and basic geometric facts.
Section \ref{sec-geom} establishes several geometric properties of smooth
projection families. 
In Section \ref{sec-localcomparison} we 
prove the
local comparison principle by comparing nonlinear fibers with thin
rectangles. 
Section \ref{sec-globalcomparison} contains the proofs of
the local and global comparison theorems together with applications to
self-similar and random constructions. Finally,
Section \ref{sec-Gfunction} develops applications to unions of level
curves, including Theorem \ref{thm-Gfunction},
Theorem \ref{thm:unioncircles}, and further examples.

\section{Geometric properties of projection families}\label{sec-geom}

We start by noting a few simple properties. First, by the linearity of the projections $\pi_\alpha$, we have
\begin{equation}\label{eq: fav linear} 
\Fav(tE) =  t\Fav(E) 
\end{equation}
for any measurable $E\subset\R^2$.
For the general mappings under consideration, we do not assume linearity and \eqref{eq: fav linear} may not hold with $\FavP$ in place of $\Fav$. However, we still have subadditivity.

\begin{lem}[Subadditivity of generalized Favard length]
Let $X_j\subset\R^2$ be measurable, and let $X=\bigcup_j X_j$. Then
$$
\FavP(X)\leq \sum_j \FavP(X_j).
$$
\end{lem}
\begin{proof}
For each $\alpha\in I$, we have $|\Phi_\alpha(X)|=|\bigcup_j \Phi_\alpha(X_j)|
\leq \sum_j |\Phi_\alpha(X_j)|$. The lemma follows upon integration in $\alpha$.
\end{proof}

\begin{lem}[Lipschitz continuity of projection length]
\label{continuity}
If $E\subset \R^2$ is a bounded set with finitely many connected components, then the map 
\[
\theta\mapsto |\pi_\theta(E)|
\]
is Lipschitz.  More precisely, if $n$ is the number of connected components and $r$ is such that $E\subset B_r(0)$, then
\[
||\pi_\theta(E)|-|\pi_\f(E)||\leq nr|\theta-\f|.
\]
\end{lem}
\begin{proof}
Let $\theta,\f\in [0,\pi]$, and suppose $|\pi_\theta(E)|\leq |\pi_\f(E)|$.  Since $E$ has $n$ connected components, $\pi_\theta(E)$ consists of $n'\leq n$ disjoint intervals.  Let $\rho_i$ be the length of the $i$-th interval.  Then, there exist rectangles $R_1,\dots,R_{n'}$ in direction $\theta$ such that $R_i$ has dimensions $r\times \rho_i$ and
\[
E\subset \bigcup_{i=1}^{n'}R_i.
\]
\begin{figure}
\centering
\begin{tikzpicture}
\draw[thick] (-1,0)--(6,0);

\draw[thick](0,2)--(4,0)--(4.5,1)--(.5,3)--(0,2);

\draw (4.2,.4) arc (60:0:.5);
\draw (3.2,.4) arc (120:180:.5);

\draw [decorate,
    decoration = {brace, amplitude=10pt}] (.75,3.5)--(4.75,1.5);
\draw [decorate,
    decoration = {brace, amplitude=10pt}] (-.5,2.25)--(0,3.25);
\draw [decorate,
    decoration = {brace, amplitude=10pt}] (3.95,-.1)--(0,-.1);
\draw [decorate,
    decoration = {brace, amplitude=10pt}] (4.5,-.1)--(4.05,-.1);

\node[above right] at (2.75,2.85) {$r$};
\node[above left] at (-.5,2.85) {$\rho_i$};
\node[below] at (2,-.5) {$r|\sin(\theta-\f)|$};
\node[below] at (5.2,-.5) {$\rho_i|\cos(\theta-\f)|$};
\node at (5,.3) {$\theta-\f$};
\node at (2,.3) {$\frac{\pi}{2}-(\theta-\f)$};

\end{tikzpicture}
\caption{The projection of the rectangle $R_i$}
  \label{fig: rectangle projection}
\end{figure}
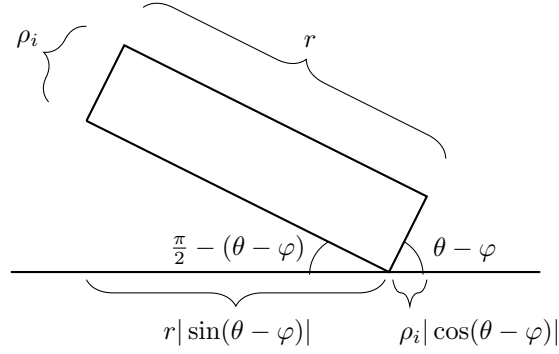
By elementary geometric considerations (Figure \ref{fig: rectangle projection}), we have
\[
|\pi_\f(R_i)|=r|\sin(\theta-\f)|+\rho_i|\cos(\theta-\f)|\leq r|\theta-\f|+\rho_i.
\]
Therefore,
\[
|\pi_\f(E)|\leq \sum_{i=1}^{n'}(r|\theta-\f|+\rho_i)\leq nr|\theta-\f|+|\pi_\theta(E)|.
\]
\end{proof}

Next, we prove that the Jacobian condition
\eqref{jacobian} implies a quantitative transversality condition, made precise in (\ref{eq:transversal}). Informally,
the directions determined by the fibers of $\Phi_\al$
change at a uniformly nondegenerate rate as $\alpha$ varies (see also \cite{BongersTaylor2023}).

\begin{lem}[$\{\Phi_\al\}$ is transversal]
\label{transversality}
If $\Phi\in C^2(I\times U)$ satisfies (\ref{gradient}) and (\ref{jacobian}), and if $\theta_\al(x)$ is as defined in Section \ref{notation}, then
\begin{equation}\label{eq:transversal}
|\partial_\al \theta_\al(x)|\approx 1.
\end{equation}
\end{lem}

\begin{proof}
Let
$$u_\al(x)=(\cos\theta_\al(x),\sin\theta_\al(x)),\ \ 
u_\al^\perp(x)= (\sin\theta_\al(x),-\cos\theta_\al(x)).
$$  
Then
\begin{align*}
\partial_\al u^\perp_\al(x)
&=(\partial_\al\theta_\al(x))u_\al(x).
\end{align*}
Since $u^\perp_\al(x)$ is orthogonal to $\nabla\Phi_\al(x)$, we have
\[
0=\nabla\Phi_\al(x)\cdot u^\perp_\al(x).
\]
Differentiating the first equation, and using that
$\nabla\Phi_\al(x)\cdot u_\al(x)=|\nabla\Phi_\al(x)|$,
we get
\begin{align*}
0&=\partial_\al\nabla\Phi_\al(x)\cdot u^\perp_\al(x)+\nabla\Phi_\al(x)\cdot (\partial_\al\theta_\al(x))u_\al(x) \\
&=\partial_\al\nabla\Phi_\al(x)\cdot u^\perp_\al(x)+(\partial_\al\theta_\al(x))|\nabla\Phi_\al(x)|.
\end{align*}
Rearranging and applying (\ref{gradient}) and (\ref{jacobian}), we have
\[
|\partial_\al\theta_\al(x)|=\frac{|\partial_\al\nabla\Phi_\al(x)\cdot u^\perp_\al(x)|}{|\nabla\Phi_\al(x)|}
=\frac{1}{|\nabla\Phi_\al(x)|^2}\left|\det
\begin{pmatrix}
\partial_1 \Phi_\al(x) & \partial_\al\partial_1 \Phi_\al(x) \\
\partial_2 \Phi_\al(x) & \partial_\al\partial_2 \Phi_\al(x)
\end{pmatrix}
\right| \approx 1.
\]
\end{proof}

\section{Local comparison of linear and nonlinear projections}\label{sec-localcomparison}
In this section, we build criteria that allow us to compare the curved projections to the orthogonal projections on small squares (locally).  This requires a number of technical ideas developed in Lemmas \ref{linesegmentbound}--\ref{fiberrectangle}, before stating our main result, Lemma \ref{CompareProjections}.

Throughout this section, we assume that $V\subset U\subset\R^2$, where $U$ is open and $V$ is compact. 
Recall the definition of $\theta_\alpha(x)$ from \eqref{defn: theta alpha}.

The scale $\delta^2$ appears throughout and arises naturally from the $C^2$ regularity of the
fibers. Indeed, a curve with uniformly bounded curvature deviates from
its tangent line by at most $O(\delta^2)$ over an interval of length
$\delta$. Hence fiber segments of length $\delta$ are naturally
contained in rectangles of dimensions $\delta\times\delta^2$, which
serve as the fundamental geometric objects in the arguments that
follow.

\begin{lem}[Bounding projections of line segments]
\label{linesegmentbound}
Let $L\subset U$ be a closed line segment in direction $\theta^\perp$, and let $\alpha\in I$.  If $\Phi\in C^2(I\times U)$ satisfies (\ref{gradient}), then
\[
\inf_{\xi\in L}|\theta_\alpha(\xi)-\theta|\lesssim \frac{|\Phi_\alpha(L)|}{|L|}\lesssim \sup_{\xi\in L}|\theta_\alpha(\xi)-\theta|.
\]
\end{lem}
\begin{proof}
Let $x,y\in L$ be such that $|\Phi_\al(x)-\Phi_\al(y)|\approx|\Phi_\al(L)|$.  By the mean value theorem, there exists $\xi\in L$ such that
\[
|\Phi_\al(L)|\approx |\nabla\Phi_\al(\xi)\cdot (x-y)|.
\]
The vector $x-y$ is in direction $\theta^\perp$ since both points lie on $L$, and the vector $\nabla\Phi_\al(\xi)$ is parallel to $\theta_\alpha(\xi)$.  We have
\begin{align*}
|\Phi_\al(L)|&\approx |\nabla\Phi_\al(\xi)||x-y|\cos(\theta_\al(\xi)-\theta^\perp) \\[.1in]
&=|\nabla\Phi_\al(\xi)||x-y|\sin(\theta_\al(\xi)-\theta).
\end{align*}
The upper bound in the lemma follows immediately by (\ref{gradient}) and the trivial bound $|x-y|\leq |L|$.  

To prove the lower bound, we may assume without loss of generality that 
\begin{equation}\label{eq:4nonzero}
\inf_{\xi\in L}|\theta_\al(\xi)-\theta|>0.
\end{equation}
Let $a,b\in U$ denote the endpoints of $L$.  We claim that if (\ref{eq:4nonzero}) holds, then the map $[0,1]\to \R$ defined by $t\mapsto \Phi_\al(a+t(b-a))$ is injective. Indeed, its derivative is 
$$
\nabla\Phi_\al(a+t(b-a)) \cdot (b-a).
$$
If this were zero at some $t_0\in(0,1)$, it would follow that $\nabla\Phi_\al(\xi)$ is perpendicular to $b-a$ at $\xi:=a+t_0(b-a)$. Since $b-a$ has direction $\theta^\perp$, 
it would follow that $\nabla\Phi_\al(\xi)$ would have direction $\theta$. But that contradicts (\ref{eq:4nonzero}), since $\theta_\al(\xi)$ is the direction of $\nabla \Phi_\al(\xi)$ and $\xi$ is a point on $L$.  

Therefore we can choose $x,y\in L$ satisfying both $|\Phi_\al(x)-\Phi_\al(y)|\approx|\Phi_\al(L)|$ and $|x-y|\approx |L|$. (In fact, we may choose $x=a$ and $y=b$.) The lower bound then follows from the same calculation as the upper bound.
\end{proof}

We will be interested in the projections of rectangles $R$ of dimensions $\de\times \de^2$, where $\de$ is a small parameter.  The upper bound in Lemma \ref{linesegmentbound}, together with the triangle inequality, implies the trivial upper bound $|\Phi_\al(R)|\lesssim \de$.  On the other hand, $R$ contains line segments $L$ of length $\de^2$ in every direction, so the lower bound in Lemma \ref{linesegmentbound} implies $|\Phi_\al(R)|\gtrsim \de^2$.  In order to achieve the lower bound, we need $(\theta_\al(z)-\theta)$ to be small throughout $z\in R$; that is, we need the long side of $R$ to be approximately tangent to the level curves of $\Phi_\al$ passing through it.  This is captured precisely in the following lemma.

\begin{lem}[``Good'' rectangles have small projections]
\label{goodrectangles}
Suppose $\Phi\in C^2(I\times U)$ satisfies (\ref{gradient}) and (\ref{seconddv}).  Let $R$ be a rectangle of dimensions $\delta\times\delta^2$ such that $B_\delta(R)\subset U$, let $\theta^\perp$ be the direction of $R$, and let $\alpha\in I$.  If $\de$ is sufficiently small and $\theta=\theta_\al(u)$ for some $u\in B_\delta(R)$, then
\[
|\Phi_\alpha(R)|\lesssim \de^2
\]
\end{lem}
\begin{proof}
Let $L$ be one of the two sides of $R$ of length $\de$.  For any $x\in L,y\in R$, let $z\in L$ be such that the segment $\overline{yz}$ is perpendicular to $L$.  Then $|y-z|\leq \de^2$, so
\[
|\Phi_\alpha(x)-\Phi_\alpha(y)|\leq |\Phi_\alpha(x)-\Phi_\alpha(z)|+O(\de^2),
\]
hence
\[
|\Phi_\alpha(R)|\leq |\Phi_\alpha(L)|+O(\de^2).
\]
By (\ref{seconddv}), the mapping $x\mapsto \theta_\al(x)$ has bounded derivatives. Hence, for any $\xi\in L$ we have 
\[
|\theta_\al(\xi)-\theta|=|\theta_\al(\xi)-\theta_\al(u)|\lesssim |\xi-u|<2\de.
\]
By Lemma \ref{linesegmentbound}, this implies $|\Phi_\al(L)|\lesssim \de^2$. The lemma follows.
\end{proof}

\begin{lem}[The neighborhood of a preimage is the preimage of a neighborhood]
\label{commutativity}
Suppose $\Phi\in C^2(I\times U)$ satisfies (\ref{gradient}), (\ref{seconddv}),  (\ref{jacobian}), and let $\al\in I,y\in\Phi_\al(V)$.  There exist constants $c,c',\epsilon_0>0$, depending only on $\Phi$ and $V$, such that for any $0<\epsilon<\epsilon_0$ we have
\[
B_{c\e}\left(\Phi_\al^{-1}(y)\right)\subset \Phi_\al^{-1}(B_\e(y))\subset B_{c'\e}\left(\Phi_\al^{-1}(y)\right).
\]
\end{lem}
\begin{proof}
If $x'\in B_{c\e}\left(\Phi_\al^{-1}(y)\right)$, then there exists $x$ such that $\Phi_\al(x)=y$ and $|x-x'|<c\e$.  By (\ref{gradient}),
we have
\[
|y-\Phi_\al(x')|=|\Phi_\al(x)-\Phi_\al(x')|\lesssim |x-x'|,
\]
so that $|y-\Phi_\al(x')|<\epsilon$ 
if $c$ is sufficiently small depending on the implicit constant above.
Hence $x'\in \Phi_\al^{-1}(B_\e(y))$, proving the first inclusion in the lemma.  

To prove the second inclusion, suppose $x'\in \Phi_\al^{-1}(B_\e(y))$, or $|y-\Phi_\al(x')|<\e$.  Without loss of generality, suppose $\Phi_\al(x')<y$ (the other case is similar).  Consider the point
\[
x''=x'+c'\e\frac{\nabla \Phi_\al(x')}{|\nabla\Phi_\al(x')|},
\]
where $c'$ will be determined later.  By the mean value theorem, there is a $\xi$ on the line segment connecting $x'$ and $x''$ which satisfies
\begin{align*}
\Phi_\al(x'')&=\Phi_\al(x')+\nabla \Phi_\al(\xi)\cdot (x''-x') \\[.1in]
&=\Phi_\al(x')+c'\e\frac{\nabla \Phi_\al(\xi)\cdot\nabla \Phi_\al(x')}{|\nabla\Phi_\al(x')|} \\[.1in]
&=\Phi_\al(x')+c'\e\left(|\nabla \Phi_\al(x')|+\frac{\nabla \Phi_\al(x')}{|\nabla\Phi_\al(x')|}\cdot(\nabla \Phi_\al(\xi)-\nabla\Phi_\al(x'))\right).
\end{align*}
We have
\[
\left|\frac{\nabla \Phi_\al(x')}{|\nabla\Phi_\al(x')|}\cdot(\nabla \Phi_\al(\xi)-\nabla\Phi_\al(x'))\right|=|\nabla \Phi_\al(\xi)-\nabla\Phi_\al(x')|
\leq C c'\epsilon,
\]
where $C$ depends only on the bounds on the second-order derivatives of $\Phi_\alpha$ in $x$. 
Let $c'=\frac{2}{\min |\nabla\Phi_\al|}$, and let $\epsilon>0$ be small enough so that
$$
Cc'\epsilon<
\frac{1}{2}\min |\nabla \Phi_\al|.
$$
Then
\[
\Phi_\al(x'')\geq\Phi_\al(x')+\frac{|\nabla \Phi_\al(x')|}{2}c'\e,
\]
and therefore
\[
\Phi_\al(x')< y<\Phi_\al(x')+\e\leq \Phi_\al(x'').
\]
We conclude that there exists $x$ between $x'$ and $x''$ such that $\Phi_\al(x)=y$. Furthermore,
\[
|x-x'|<|x''-x'|=c'\e,
\]
hence
\[
x'\in B_{c'\e}(x)\subset B_{c'\e}(\Phi_\al^{-1}(y)).
\]
\end{proof}

\begin{lem}[Fibers are contained in rectangles]
\label{fiberrectangle}
Assume $\Phi\in C^2(I\times U)$ satisfies (\ref{gradient}) and (\ref{seconddv}).  Let $\alpha\in I,y\in \Phi_\al(U)$, let $Q\subset V$ be a square of side length $\de$, and let $\Gamma$ be a connected component of $Q\cap \Phi_\alpha^{-1}(y)$.  If $\delta$ is sufficiently small (depending only on $\Phi,V,I$), then for any $\theta\in \theta_\al(Q)$ there exists a rectangle $R$ in direction $\theta^\perp$ of dimensions $\approx \de\times\de^2$ such that $\Gamma\subset R$.
\end{lem}
\begin{proof}
For notational simplicity, let $\pi=\pi_{0}$ denote projection onto the horizontal axis.  By (\ref{gradient}), we may assume without loss of generality that there exists $x_0\in Q$ and such that $\partial_2\Phi_\al(x_0)\gtrsim 1$, with constant independent of $Q$.  If $\de$ is sufficiently small, then by uniform continuity we have $\partial_2\Phi_\al\gtrsim 1$ on $Q$.  This implies that $\Gamma$ is the graph of some function $f:\pi(\Gamma)\to \R$.  Indeed, if this were not the case, then there would exist $x,x'\in \Gamma$ with $\pi(x)=\pi(x')$, and this would imply $\partial_2\Phi_\al(\xi)=0$ for some $\xi\in Q$, a contradiction.  We observe that since $\Gamma$ is compact and connected, $\pi(\Gamma)$ is a compact interval of length at most $\de$, which we denote by $J$.  Finally, we note that the uniqueness part of the implicit function theorem implies that $f\in C^2(J)$, and
\[
f'(t)=\frac{-\partial_1\Phi_\al(t,f(t))}{\partial_2\Phi_\al(t,f(t))},
\]
hence $|f'|\lesssim 1$.  
By computing the second derivative directly and applying the bounds $|\partial_2\Phi_\al(t,f(t))|\gtrsim 1$ and $|f'|\lesssim 1$ together with (\ref{gradient}) and (\ref{seconddv}), we have $|f''(t)|\lesssim 1$.  
Since $|f''|\lesssim 1$, Taylor's theorem implies that if $T$
denotes the tangent line to the graph at any point of $\Gamma$, then
\[
|f(t)-T(t)|\lesssim |J|^2,
\qquad t\in J.
\]
Hence the graph deviates from its tangent line by at most
$O(|J|^2)$, which implies
there is a rectangle $R_0$ in the direction $\theta^\perp_\al(x)$ of dimensions $\approx |J|\times |J|^2$ which contains $\Gamma$ (since $|J|\leq \de$, we can take $R_0$ to have dimensions $\approx \de\times \de^2$ if desired).  Finally, if $\theta=\theta_\al(x')$ for any $x'\in U$ with $|x-x'|\lesssim \de$ (in particular, for any $x'\in Q$), then 
\[
|\theta-\theta_\al(x)|=\left|\frac{\nabla\Phi_\al(x')}{|\nabla\Phi_\al(x')|}-\frac{\nabla\Phi_\al(x)}{|\nabla\Phi_\al(x)|}\right|\lesssim \de.
\]
Rotating $R_0$ about one of its vertices by angle $\theta-\theta_\al(x)$ and multiplying the lengths on each side by a constant factor, we obtain a rectangle $R\supset R_0$ with dimensions $\approx \de\times\de^2$ in direction $\theta$.
\end{proof}

We now arrive at the heart of this section. 
The proof of Lemma \ref{CompareProjections}
proceeds by passing from one projection family to the other
through coverings by thin rectangles.

To prove the upper bound, we begin with a covering of
$\pi_\theta(E)$ by intervals. 
Pulling these intervals back through $\pi_\theta$
produces thin strips, which are then localized to thin
rectangles.
Since the long sides of these
rectangles are approximately tangent to the fibers of
$\Phi_\alpha$, Lemma \ref{goodrectangles} implies that their
$\Phi_\alpha$-projections are small.

The lower bound follows a similar strategy, but we begin instead
with a covering of $\Phi_\alpha(E)$. 
Here an additional difficulty arises because the fibers of $\Phi_\alpha$ are no
longer parallel lines. Using the Implicit Function Theorem, we
show that each connected component of a fiber may be locally
contained in a thin rectangle whose long side is approximately
tangent to that fiber. This then allows us to
reverse the argument above and compare the orthogonal
projections with the generalized projections.

\begin{lem}
\label{CompareProjections}
Assume that $\Phi\in C^2(I\times U)$ satisfies (\ref{gradient}) and (\ref{seconddv}). There exists $\de_0>0$ (depending only on $\Phi,V,I$) such that for any $0<\de<\de_0$, the following holds.  Let $Q\subset V$ be a compact square of side length $\de$, let $\cQ$ be a finite collection of squares which are contained in $Q$ and have side length $\de^2$, and let $E=\bigcup\cQ$.  For any $\alpha\in I$ and $\theta\in \theta_\al(Q)$,  we have
\[
|\Phi_\alpha(E)|\lesssim |\pi_{\theta}(E)|.
\]
If in addition $\Phi$ satisfies (\ref{jacobian}), and (\ref{components}), then
\[
|\Phi_\alpha(E)|\approx |\pi_{\theta}(E)|.
\]
\end{lem}
\begin{proof}
Assume (\ref{gradient}) and (\ref{seconddv}).  We first prove the bound $|\Phi_\alpha(E)|\lesssim |\pi_{\theta}(E)|$.  We have
\begin{equation}
\label{pidecomposition}
\pi_{\theta}(E)=\bigcup_{Q'\in \cQ}\pi_\theta(Q'),
\end{equation}
where each $\pi_\theta(Q')$ is an interval of length $|\pi_\theta(Q')|\approx \de^2$.  
By the Vitali covering lemma, we may choose a subcollection $\{J_1,\dots,J_n\}$ of the collection $\{\pi_\theta(Q'):Q'\in \cQ\}$ consisting of intervals which are disjoint and satisfy
\[
\pi_\theta(E)\subset \bigcup_{k=1}^n 3J_k,
\]
where $3J_k$ is the interval with the same center as $J_k$ and triple the length.  
We also have, by (\ref{pidecomposition}), that
 \[
 \pi_\theta(E)\supset \bigsqcup_{k=1}^n J_k,
 \]
 with the union on the right being disjoint. 
It follows that
\[
\sum_{k=1}^n |J_k|\leq |\pi_\theta(E)|\leq \sum_{k=1}^n |3J_k|\approx \sum_{k=1}^n |J_k|,
\]
and so
\begin{equation}
\label{sizeofpi}
|\pi_\theta(E)|\approx \sum_{k=1}^n |J_k|\approx n\de^2.
\end{equation}
Each of the sets $\pi_\theta^{-1}(J_k)$ is an infinite strip in direction $\theta^\perp$ of width $|J_k|\approx \de^2$.  
Therefore, for each $k$, there exists a rectangle $R_k$ of dimensions $\approx \de\times \de^2$ which contains $\pi_\theta^{-1}(J_k)\cap Q$, and such that the family of rectangles $\{R_1,\dots,R_n\}$ is pairwise disjoint and pairwise parallel.  

Since $R_k$ contains $\pi_\theta^{-1}(J_k)\cap Q$, the enlarged
rectangle $3R_k$ contains $\pi_\theta^{-1}(3J_k)\cap Q$. Combining
this with the covering
\[
\pi_\theta(E)\subset \bigcup_{k=1}^n 3J_k
\]
yields
\[
E\subset \bigcup_{k=1}^n 3R_k,
\]
where $3R_k$ is a rectangle with the same center and same side length of the $\de$ side, with triple the length on the $\de^2$ side.  
By Lemma \ref{goodrectangles} and equation (\ref{sizeofpi}), we therefore have
\begin{align*}
|\Phi_\al(E)|&\leq \left|\Phi_\al\left(\bigcup_{k=1}^n 3R_k\right)\right| \\
&\leq \sum_{k=1}^n |\Phi_\al(3R_k)| \\
&\lesssim \sum_{k=1}^n \de^2 \\
&\approx |\pi_\theta(E)|.
\end{align*}

Next, we assume (\ref{gradient}), (\ref{seconddv}), (\ref{jacobian}), (\ref{components}), and we prove the reverse inequality using a similar argument, but we instead start with a covering of $\Phi_\alpha(E)$.  Observe 
\begin{equation}
    \label{phidecomposition}
\Phi_\alpha(E)=\bigcup_{Q'\in \cQ}\Phi_\al(Q'),
\end{equation}
and observe that each set $\Phi_\al(Q')$ is an interval.

We first prove that
$|\Phi_\al(Q')|\approx \delta^2$
for each square $Q'\in\mathcal Q$.
 Indeed, the upper bound $|\Phi_\al(Q')|\lesssim \de^2$ follows since $\Phi_\alpha$ is uniformly Lipschitz by (\ref{gradient}). On the other hand, $Q'$ contains a line segment $L$ of length $\approx\de^2$ in the direction $\theta$. If $\de$ is sufficiently small, we have
$|\theta_\al(\xi)-\theta^\perp|\geq 0.1$ for all $\xi\in Q'$. By Lemma \ref{linesegmentbound}, we have $|\Phi_\al(Q')|\gtrsim  \de^2$, and the claim is proved.

By the Vitali covering lemma, there is a subcollection $\{J_1,\dots,J_m\}\subset \{\Phi_\al(Q'):Q'\in \cQ,Q'\subset Q\}$ which is pairwise disjoint and satisfies
\begin{equation}
    \label{vitaliphi}
\bigcup_{k=1}^m J_k\subset \Phi_\al(E)\subset \bigcup_{j=1}^m 3J_k.
\end{equation}
The first inclusion in (\ref{vitaliphi}) implies
\begin{equation}
\label{sizeofphi}
|\Phi_\al(E)|\geq \sum_{k=1}^m |J_k|\approx m\de^2.
\end{equation}
For each $k$, let $y_k$ denote the center of $J_k$ (and hence also the center of $3J_k$).  By Lemmas \ref{fiberrectangle} and (\ref{components}), there exists a family $\cR_k$ of $O(1)$ rectangles in direction $\theta^\perp$, each of dimensions $\approx \de\times \de^2$, such that $\bigcup \cR_k\supset \Phi_\al^{-1}(y_k)\cap Q$. Let $C\cR_k=\{CR:R\in \cR_k\}$, where $CR$ denotes a rectangle centered at the same point and in the same direction as $R$, with dimensions increased by some large factor $C=O(1)$. By Lemma \ref{commutativity}, we have $\bigcup(C\cR_k)\supset \Phi_\al^{-1}(3J_k)\cap Q$ if $C$ is large enough.

If $x\in Q'$ for some $Q'\in\cQ$, then $\Phi_\al(x)\in 3J_k$ for some $k$, hence $Q'\subset \Phi_\al^{-1}(3J_k)\cap Q\subset \bigcup(C\cR_k)$.  By (\ref{sizeofphi}),
\[
|\pi_\theta(E)|\leq \sum_{k=1}^m\sum_{R\in\cR_k} |\pi_\theta(CR)|\lesssim m\de^2 \lesssim |\Phi_\al(E)|.
\]
\end{proof}

\section{Proof of the comparison theorems}\label{sec-globalcomparison}

We first prove the local comparison theorems and then deduce
the global comparison theorem for self-similar sets.

\begin{proof}[Proof of Theorem \ref{thm: mainlocal}]
First we assume (\ref{gradient}), (\ref{seconddv}), and (\ref{jacobian}), and prove the upper bound.  Let $\de_0$ be sufficiently small that Lemma \ref{CompareProjections} applies, let $Q$ be a compact square of side length $\de<\de_0$, and let $E\subset Q$ be a union of $N$ compact squares of side length $\de^2$.  For $j=0,1,\dots,n$ define
\[
\theta_j=\frac{j\pi}{n},
\]
where $n$ is a large parameter to be defined later.  By Lemma \ref{continuity}, if $n$ is sufficiently large, then
\[
\frac{1}{n}\sum_{j=1}^n |\pi_{\theta_j}(E)|\approx \Fav(E).
\]
Fix a point $x_0\in Q$.  For each $j$, define
\[
A_j=\{\alpha:\theta_\alpha(x_0)\in(\theta_{j-1},\theta_j]\}.
\]
Clearly, the sets $A_j$ partition the interval $I$.  By Lemma \ref{transversality}, the set $A_j$ is a union of $O(1)$ intervals of length $\approx \frac{1}{n}$ for each $j$.  By Lemma \ref{CompareProjections}, we have 
\begin{equation}
\label{upperbound}
|\Phi_\alpha(E)|\lesssim |\pi_{\theta_\al(x_0)}(E)| \hbox{ for each }\al\in A_j.
\end{equation}
We claim that the right side is $\lesssim |\pi_{\theta_j}(E)|$ if $n$ is large enough (depending on $N$ and $\de$). 
Indeed, by Lemma \ref{continuity}
\[
|\pi_{\theta_\al(x_0)}(E)| = |\pi_{\theta_j}(E)| + O\left( \frac{N}{n}\right).
\]
On the other hand, we have $|\pi_\theta(E)|\geq \de^2$ for every $\theta$, so that the main term above dominates the error term $O\left( \frac{N}{n}\right)$ if $n$ is sufficiently large. 

It follows that $|\Phi_\alpha(E)|\lesssim |\pi_{\theta_j}(E)|$ for each $\al\in A_j$, as claimed. Integrating on $A_j$, and using that $|A_j|\approx \frac{1}{n}$, we get
    \[
    \int_{A_j} |\Phi_\alpha(E)|d\alpha\lesssim\frac{1}{n}|\pi_{\theta_j}(E)|.
    \]
    Summing in $j$,
    \[
    \FavP(E)\lesssim \frac{1}{n}\sum_j|\pi_{\theta_j}(E)|\approx\Fav(E).
    \]
\end{proof}

\begin{proof}[Proof of Theorem \ref{thm: mainlocal-lower}]
Assume now that in addition to (\ref{gradient}), (\ref{seconddv}), and (\ref{jacobian}), we also have (\ref{components}) and (\ref{angle-range}). 
Let $\de_0,Q,E$ be as in the proof of Theorem \ref{thm: mainlocal}. Let also $\Theta=[t,t']$, and let $x_0\in Q$ be as in (\ref{angle-range}). For $j=0,1,\dots,n$, where $n$ is a large parameter to be chosen later, define
\[
\theta_j:=t+ \frac{j(t'-t)}{n},\ \ 
A_j:=\{\alpha:\theta_\alpha(x_0)\in(\theta_{j-1},\theta_j]\}.
\]
Then $\bigcup A_j=\Theta$ and, by Lemma \ref{transversality}, each $A_j$ is a union of $O(1)$ intervals of length $\approx \frac{|\Theta|}{n}$ for each $j$. By the same argument as in the proof of Theorem \ref{thm: mainlocal}, but with $\approx$ in (\ref{upperbound}) as provided by Lemma \ref{CompareProjections}, we have 
$$|\Phi_\alpha(E)|\approx  |\pi_{\theta_j}(E)|$$
for each $\al\in A_j$.
Integrating and then summing in $j$, we get
\begin{align*}
\FavP(E)&\gtrsim \sum_j \int_{A_j} |\Phi_\alpha(E)|d\alpha
\\
&\approx \sum_j \frac{|\Theta|}{n}|\pi_{\theta_j}(E)|
\\
&\approx \int_\Theta |\pi_\theta(E)|
\end{align*}
provided that $n$ was chosen large enough.
\end{proof}

\begin{proof}[Proof of Theorem \ref{thm:comparison}]
Let $E_n$ and $\de_0$ be as in Theorem \ref{thm:comparison}. Let $B=\max|b_i|$. We note that $E_n$ and $E$ are contained in the square $[-R,R]^2$, where
$$
R=(B+L^{-1})\sum_{j=0}^\infty L^{-j}\leq 2(B+1).
$$
Let also $m=\lfloor n/2\rfloor$, so that $n-m$ is equal to either $m$ or $m+1$ depending on the parity of $n$. 
Then $E_{n}$ is a union of $L^{n-m}$ similar copies $E_{m,j}$ of $E_m$, rescaled by a factor of $L^{-(n-m)}$ and is therefore contained in a square of side length at most $4(B+1)L^{-(n-m)}\leq 4(B+1)L^{-m}$.

Let $m$ be sufficiently large so that $\de:=4(B+1)L^{-m} <\de_0$. For $j=1,\dots,L^{n-m}$, let $Q_j$ be a $\delta$-square containing $E_{m,j}$.
Note that, with the choice $m=\lfloor n/2\rfloor$, each $E_{m,j}$ is contained in a $\delta$-square and is a union of $\delta^2$-squares, as required by Theorem \ref{thm: mainlocal}.
For each $j$, we apply Theorem \ref{thm: mainlocal} to the set $E_{m,j}$ to get the bound
\[
\FavP(E_{m,j})\lesssim \Fav(E_{m,j}).
\]
By \eqref{eq: fav linear},
\[
\Fav(E_{m,j})=\de\Fav(E_m).
\]
Combining these observations, and using the subadditivity of $\FavP$, we get
\begin{equation}\label{eq:subadd}
\begin{split}
\FavP(E_{n} )&\leq \sum_{j} \FavP(E_{m,j}) \\
&\lesssim \sum_{j} \Fav(E_{m,j}) \\
&= \de\sum_{j}\Fav(E_m) \\
&\approx \Fav(E_m), \\
\end{split}
\end{equation}
proving the upper bound of the theorem.

Lower bounds do not follow in the same manner since
the first inequality in (\ref{eq:subadd}) (subadditivity) is not reversible. 
Instead, we use an energy argument from \cite{BongersTaylor2023}. 
In preparation for that argument,
recall the condition (\ref{jacobian}):
$$
\left|
\det
\begin{pmatrix}
\partial_1 \Phi_\al(x) & \partial_\al\partial_1 \Phi_\al(x) \\
\partial_2 \Phi_\al(x) & \partial_\al\partial_2 \Phi_\al(x)
\end{pmatrix}
\right|\approx 1.
$$
Geometrically, the absolute value of the determinant in this condition is equal to the length of the cross product,
$$
|\nabla \Phi_\alpha \times \partial_\alpha \nabla \Phi_\alpha|
=|\nabla \Phi_\alpha|\,| \partial_\alpha \nabla \Phi_\alpha|\,\sin\beta_\alpha,
$$
where $\beta_\alpha$ is the angle between the two vectors. Thus there is a constant $0<c_1<1$ such that
\begin{equation}\label{lower-c2}
|\nabla \Phi_\alpha|\geq c_1,\ | \partial_\alpha \nabla \Phi_\alpha|\geq c_1,\ \sin\beta_\alpha\geq c_1.
\end{equation}
Let $0<\epsilon<c_1/20$ be small enough (to be chosen later). 
We may choose small convex sets $U'\subset U$ and $I'\subset I$ so that:
\begin{itemize}
\item $U'\cap E_n$ are approximations of a $1$-set,
\item for each $\alpha \in I'$, $\Phi_\al(x)$ is defined for each $x\in U'$, 
\item there are vectors $v,w\in \R^2$ such that
\begin{equation*}
|\nabla \Phi_\alpha -v|<\epsilon,\ \ 
   |\partial_\alpha \nabla \Phi_\alpha -w|<\epsilon, 
\end{equation*}
for all $\alpha\in I'$ and $x\in U'$. 
\end{itemize}
The last condition, together with (\ref{lower-c2}), implies that
\begin{equation}\label{lower-c1}
|v|\geq c_1/2,\ |w|\geq c_1/2,\ \sin\angle(v,w)\geq c_1/2.
\end{equation}

We would like to apply \cite[Theorem 1.5]{BongersTaylor2023} with $m=1$ and 
$$\{\widetilde{\pi_\al}: \al\in A\} = \{\Phi_\alpha: \alpha\in I\}.$$ 
To do this, we need to verify the transversality condition of \cite[Definition 1.3]{BongersTaylor2023} with $m=1$. In our notation, with $x$ and $\alpha$ restricted to $U'$ and $I'$ respectively, this condition says that 
there is a constant $c>0$ and a $\delta_0>0$ so that if $\delta\in (0,\delta_0]$ and $x,y\in U'$, then 
\begin{equation}\label{BT-condition}
\left| \left\{\alpha\in I': \frac{|\Phi_\alpha(x)-\Phi_\alpha(y)|}{|x-y|}<\delta \right\} \right|
\leq c\delta.
\end{equation}
Fix $x,y\in U'$. First, we use the (1-dimensional) mean value theorem to write
\begin{equation}\label{eq:mvt}
\frac{|\Phi_\alpha(x)-\Phi_\alpha(y)|}{|x-y|}=
|D_{u}\Phi_\alpha(z)|,
\end{equation}
where $u=\frac{x-y}{|x-y|}$ is the unit vector in the direction of $x-y$, $z=z(\alpha)$ is some point on the line segment from $x$ to $y$, and $D_{u}\Phi_\alpha(z)$ is the directional derivative at that point. We have
$$
D_{u}\Phi_\alpha(z)= \nabla \Phi_\alpha(z)\cdot u.
$$
Suppose that the set in (\ref{BT-condition}) is nonempty, so that $|\Phi_{\alpha_0}(x)-\Phi_{\alpha_0}(y)|<\delta|x-y|$ for some $\alpha_0\in I'$ and $\delta<\epsilon$. Then, with $z=z(\alpha_0)$,
$$
|v\cdot u| \leq |\nabla \Phi_{\alpha_0}(z)\cdot u|+\epsilon \leq 2\epsilon.
$$
In particular, 
$$
|\cos\angle(u,v) |\leq \frac{2\epsilon}{|v|}\leq \frac{4 \epsilon}{c_1}.$$
At the last step, we used the first part of (\ref{lower-c1}). Thus, $u$ and $v$ must be close to being perpendicular.
On the other hand, the last part of (\ref{lower-c1}) implies that $v$ and $w$ are separated in angle by at least $c_1/2$. Assuming that $\epsilon$ is small enough relative to $c_1$, it follows that $u$ and $w$ cannot be too close to being perpendicular. Quantitatively, 
\begin{equation}\label{lower-cos2}
|\cos\angle(u,w)|\geq \frac{c_1}{4},
\end{equation}
assuming that $\epsilon$ was chosen small enough.

By the mean value theorem again, for all $\alpha\in I'$ we have
\begin{align*}
&\left| \partial_\alpha \frac{\Phi_\alpha(x)-\Phi_\alpha(y)}{|x-y|}\right|
= \frac{|\partial_\alpha \Phi_\alpha(x)-\partial_\alpha \Phi_\alpha(y)|}{|x-y|}
\\
&= |D_u \partial_\alpha \Phi_\alpha(z')| 
= |\partial_\alpha \nabla \Phi_\alpha(z')\cdot u|,
\end{align*}
where $z'$ (not necessarily the same as $z$, and possibly depending on $\alpha$) lies on the line segment from $x$ to $y$. We have
$$
|\partial_\alpha \nabla \Phi_\alpha(z')\cdot u|
\geq |w\cdot u|-\epsilon \geq |w||\cos\angle(u,w) -\epsilon 
\geq \frac{c_1}{2}\cdot \frac{c_1}{4} -\epsilon 
\gtrsim 1.
$$
Thus, if the set in (\ref{BT-condition}) is nonempty for some $x,y\in U'$, it is contained in an interval of length $\lesssim \delta$.

With the transversality condition in \eqref{BT-condition} confirmed for all $\alpha\in I'$ and $x,y\in U'$, we now apply \cite[Theorem 1.5]{BongersTaylor2023} to the set $E_n\cap U'$. We get
$$\FavP(E_n\cap U')=\int_{I'} |\Phi_\al(E_n\cap U')|d\al \gtrsim \frac{1}{n}.$$
Since $\FavP(E_n)\geq \FavP(E_n\cap U')$, our lower bound follows.

\end{proof}

\begin{proof}[Proof of Theorem \ref{thm: CSS generalization}]
Let $L\geq 2$, and let $\{S_n\}_n$ be the sequence of sets constructed as described in Section \ref{subsec-prob_constructions}.  For each $n$, let $\cS_n\subset \cD_n$ be the set of squares making up $S_n$ (that is, $S_n=\bigcup\cS_n$). Thus, each $\cS_n$ contains $L^n$ squares of side length $L^{-n}$.  For any $n\in 2\N$ and any $Q\in \cD_{n/2}$, define $E_{n,Q}:=S_n\cap Q$.  Note that our notation does not require $Q\in\cS_n$; if $Q\in\cS_{n/2}$, then $E_{n,Q}$ is a finite union of squares of side length $L^{-n}$, and otherwise it is empty.

We start with a couple of simple probabilistic calculations.  We denote by $\mathbb{P}(A|B)$ the conditional probability of event $A$ given event $B$, and by $\E(X\,|\,B)$ the conditional expectation of random variable $X$ given $B$.  Our first observation is that for any $Q\in \cS_{n/2}$, the set $E_{n,Q}$ is (up to translation) a copy of another random realization of $S_{n/2}$ which has been scaled by a factor $L^{-n/2}$.  We therefore have
\begin{equation}
\label{eq: scale CSS}
\E[\Fav(E_{n,Q})\,|\,Q\in \cS_{n/2}]=L^{-n/2}\E[\Fav(S_{n/2})].
\end{equation}
Next, recall that our uniformity assumption on the probability distribution $\chi$ states that for any $Q\in \cD_1$, we have $\mathbb{P} (Q\in S_1)=L^{-1}$.  If $n>1$ and $Q\in \cD_n$, let $Q'\in \cD_{n-1}$ denote the ``parent'' of $Q$; that is, the unique square in $\cD_{n-1}$ containing $Q$ as a subset.  For $0\leq i<n$, define $Q^{(i)}\in \cD_{n-i}$ recursively by
\begin{align*}
Q^{(0)}&=Q, \\
Q^{(i+1)}&=(Q^{(i)})'.
\end{align*}
It follows from the definition of uniformity, and the assumption that the same distribution $\chi$ is used at each stage of the construction, that for any $n$ and any $Q\in \cD_n$ we have
\begin{equation}
\label{eq: probability calculation}
\mathbb{P}(Q\in \cS_n)=\mathbb{P}(Q^{(n-1)}\in \cS_1)\prod_{i=0}^{n-2}\mathbb{P}(Q^{(i)}\in \cS_{n-i}\,|\,Q^{(i+1)}\in \cS_{n-i-1})=L^{-n}.
\end{equation}
With these calculations in place, we are ready to complete the proof.  Let $\de_0$ be as in Theorem \ref{thm: mainlocal} with $V=[0,1]^2$, and fix $n_0$ be such that $L^{-n_0/2}<\de_0$.  Note that $\de_0$ depends only on $\Phi$ and $I$, so $n_0$ depends only on $\Phi,I,L$.  Let $n>n_0$ be even, and let $\de:=L^{-n/2}<\de_0$.   For any $Q\in\cS_{n/2}$, the set $E_{n,Q}$ is a finite union of squares of side length $L^{-n}=\de^2$, hence Theorem \ref{thm: mainlocal} applies.  On the other hand, if $Q\in \cD_{n/2}\setminus \cS_{n/2}$ then $E_{n,Q}$ is empty.  In either case, we have $\FavP(E_{n,Q})\lesssim \Fav(E_{n,Q})$.  This gives
\begin{equation}
\label{eq: conditional expectation}
\begin{split}
\E[\FavP(S_n)]&\leq \E[\sum_{Q\in\cS_{n/2}}\FavP(E_{n,Q})] \\
&=\E[\sum_{Q\in\cD_{n/2}}\textbf{1}_{S_{n/2}}(Q)\FavP(E_{n,Q})] \\
&\overset{(\ref{eq: probability calculation})}{=}L^{-n/2}\sum_{Q\in \cD_{n/2}}\E[\FavP(E_{n,Q})\,|\,Q\in \cS_{n/2}] \\
&\overset{\textup{Theorem } \ref{thm: mainlocal}}{\lesssim} L^{-n/2}\sum_{Q\in \cD_{n/2}}\E[\Fav(E_{n,Q})\,|\,Q\in \cS_{n/2}] \\
&\overset{(\ref{eq: scale CSS})}{=}L^{-n}\sum_{Q\in \cD_{n/2}}\E[\Fav(S_{n/2})] \\
&=\E[\Fav(S_{n/2})].
\end{split}
\end{equation}
This proves the result for all even $n>n_0$.  For odd $n>n_0$, we observe that $S_n\subset S_{n-1}$, hence (\ref{eq: conditional expectation}) implies
\[
\E[\FavP(S_n)]\leq \E[\FavP(S_{n-1})]\lesssim \E[\Fav(S_{(n-1)/2})]=\E[\Fav(S_{\lfloor n/2\rfloor})].
\]
Finally, for $n\leq n_0$, we note that $\FavP(S_n)\leq 1$ and $\Fav(S_{n/2})\geq L^{-n/2}\geq L^{-n_0/2}$.  Since our constant is allowed to depend on $\Phi,I,$ and $L$, and since $n_0$ depends only on these parameters, we may extend the estimate $\E[\FavP(S_n)]\lesssim \E[\Fav(S_{\lfloor n/2\rfloor})]$ to all $n$ by ensuring the implied constant is at least $L^{n_0/2}$.
\end{proof}

\section{Unions of curves}\label{sec-Gfunction}

We now state and prove our main results on unions of curves. We start with a very general result in Theorem \ref{thm-Gfunction} below, then deduce Theorem \ref{thm:unioncircles} as a consequence. Theorem \ref{thm-Gfunction} allows many other applications; as an example, we state and prove a similar result on unions of families of ellipses with slowly varying major and minor axes.

\begin{thm}\label{thm-Gfunction}
Let $\mathcal{U}\subset\R^4$ be an open and bounded set, and let $G: \mathcal{U}\to \R$ be a $C^2$ function. Assume that
$G$ and its first and second derivatives are bounded uniformly on $\mathcal{U}$. Assume further that
\begin{equation}\label{Ggrady}
|\nabla_y G(x,y)|\gtrsim 1 \hbox{ on }\mathcal{U},
\end{equation}
\begin{equation}
\label{jacobian-G}
\left|
\det
\begin{pmatrix}
 0 & G_3 & G_4\\
 G_1& G_{31} & G_{41}\\
 G_2 & G_{32} & G_{42}
\end{pmatrix}
\right|\approx 1 \hbox{ on }\mathcal{U},
\end{equation}
where $G_i$ denotes the derivative of $G$ in the $i$-th variable, and $G_{ij}$ are the corresponding second-order derivatives.
For each $x=(x_1,x_2)\in\R^2$, let
$$
\Gamma_x=\{y=(y_1,y_2)\in\R^2:\ G(x,y)=0\}.
$$
By (\ref{Ggrady}) and the Implicit Function Theorem, each $\Gamma_x$ is a curve or union of curves in $\R^2$.

Let $E_n\subset\R^2$ be the $n$-th iteration of a self-similar 1-set $E$ of the form defined in Section \ref{subsec-selfsim}. Assume further that $\{(x,y):\ x\in E_n,\ G(x,y)=0\}\subset\mathcal{V}$ for all $n$, where
$\mathcal{V}\subset\mathcal{U}$  is a fixed compact set. 
Define
\begin{equation}\label{eq:gammaunions}
\mathcal{G}(E_n):=\bigcup_{x\in E_n} \Gamma_x \subset\R^2,
\end{equation}
and similarly for $\mathcal{G}(E)$. Then 
\begin{equation}\label{eq:Gupperbound}
|\mathcal{G}(E_n)|\lesssim \Fav(E_{\lfloor n/2\rfloor} ),
\end{equation}
$$
|\mathcal{G}(E)|=\Fav(E)=0.
$$

\end{thm}

\begin{proof}[Proof of Theorem \ref{thm-Gfunction}]
The proof is by reducing to an application of Theorem \ref{thm:comparison}. We first split $\mathcal{U}$ into smaller sets where the theorem can be applied. 
By \eqref{Ggrady}, either $G_4(x,y)$ or $G_3(x,y)$ is bounded away from $0$ for each $(x,y)$ in $\mathcal{U}$, where $G_i$ denotes the partial derivative of $G$ in $y_i$.

Assume first that
\begin{equation}\label{eq:partial4}
|G_4(x,y)|\gtrsim 1
\end{equation}
holds at some point $(x,y)\in\mathcal{U}$.
 Then, we can solve locally for $y_2$ as a function of the remaining variables, so that there exists a $C^2$ function $y_2=\Phi(y_1,x)$ 
such that
\begin{equation}\label{Gimplicit}
G(x_1,x_2,y_1,\Phi(y_1,x))=0.
\end{equation}


Furthermore, $\Phi$ satisfies the assumptions \eqref{gradient} and \eqref{seconddv} of Theorem \ref{thm: mainlocal}, with $\alpha=y_1$. Indeed,
\[
\partial_i\Phi=-\frac{G_i}{G_4}, \qquad i=1,2.
\]
Since the derivatives of $G$ are uniformly bounded and $|G_4|\gtrsim1$ on the neighbourhood under consideration, we have
\[
\|\nabla_x\Phi\|\lesssim 1.
\]
On the other hand, expanding the determinant in \eqref{jacobian-G} shows that
\[
1\lesssim
\left|
\det
\begin{pmatrix}
0&G_3&G_4\\
G_1&G_{31}&G_{41}\\
G_2&G_{32}&G_{42}
\end{pmatrix}
\right|
\lesssim
\sqrt{G_1^2+G_2^2},
\]
where the second inequality follows from the boundedness of the first and second derivatives of $G$. Hence
\[
\sqrt{G_1^2+G_2^2}\gtrsim 1.
\]
Since $|G_4|\approx1$, it follows that
\[
\|\nabla_x\Phi\|
=
\frac{\sqrt{G_1^2+G_2^2}}{|G_4|}
\approx1,
\]
proving \eqref{gradient}. Condition \eqref{seconddv} follows similarly from the implicit differentiation formulas for the second derivatives of $\Phi$, using the boundedness of the first and second derivatives of $G$ together with $|G_4|\gtrsim1$.

To verify (2.3), differentiate (6.6) with respect to $x_i$, $i=1,2$, to obtain
\[
\partial_i \Phi=-\frac{G_i}{G_4}.
\]
Differentiating once more with respect to $y_1$ and using
\[
\partial_{y_1}\Phi=-\frac{G_3}{G_4},
\]
a direct computation yields
\[
\det
\begin{pmatrix}
\partial_1\Phi & \partial_{y_1}\partial_1\Phi\\
\partial_2\Phi & \partial_{y_1}\partial_2\Phi
\end{pmatrix}
=
\frac{1}{G_4^4}
\det
\begin{pmatrix}
0 & G_3 & G_4\\
G_1 & G_{31} & G_{41}\\
G_2 & G_{32} & G_{42}
\end{pmatrix}.
\]
Since (6.5) implies $|G_4|\gtrsim 1$ on the neighbourhood under consideration, condition (2.3) is equivalent to (6.2).

Fix the compact set $\mathcal{V}\subset\mathcal{U}$. Let $(x,y)\in\mathcal{V}$. By (\ref{Ggrady}), at least one of $|G_3(x,y)|\gtrsim 1 $ and $|G_4(x,y)|\gtrsim 1$ holds. Assume without loss of generality that we have the latter. This means that there is some neighbourhood $\mathcal{U}_{x,y}\subset\mathcal{U}$ where we can solve (\ref{Gimplicit}) for $y_2=\Phi(y_1,x)$ as above. Making this neighbourhood smaller if necessary, while still keeping $(x,y)\in \mathcal{U}_{x,y}$, we may assume that $\mathcal{U}_{x,y}=U\times I\times J$, where $x\in U\subset\R^2$, $y_1\in I\subset\R$, and $y_2\in J\subset\R$. 

Since $\mathcal{V}$ is compact, we may cover it by $O(1)$ such neighbourhoods $\mathcal{U}_j$, defined either exactly as above or with the $y_1$ and $y_2$ variables interchanged. Let $E_n\subset\R^2$ be as in the assumptions of the theorem, and let
$$
\mathcal{G}_j(E_n)=\bigcup_{x\in E_n} \{y:\ (x,y)\in \mathcal{U}_j,\ G(x,y)=0\},
$$
then
$$
\mathcal{G}(E_n)\subset\bigcup_j \mathcal{G}_j(E_n).
$$
Hence, to estimate $|\mathcal{G}(E_n)|$ from either above or below, it suffices to give  similar bounds for each $|\mathcal{G}_j(E_n)|$. 

We now fix $j$. Without loss of generality, we assume that (\ref{eq:partial4}) holds on $\mathcal{U}_j$, and that each $\Gamma_x$ is given by $y_2=\Phi(y_1,x)$ with $\Phi$ defined as above. We further have $\mathcal{U}_{j}=U\times I\times J$, where $U\subset\R^2$ is an open set, and $I,J\subset\R$ are intervals. 
This places us in the set-up of Theorem \ref{thm:comparison}. We write
\begin{align*}
|\mathcal{G}_j(E_n)|&=\int_{\mathcal{G}_j(E_n) }1 dy\\
&= \int_I \Big|\Big\{y_2:\ y_2=\Phi_{y_1}(x)\hbox{ for some }x\in E_n\Big\}\Big|dy_1
\\
&= \int_I \Big|\Phi_{y_1}(E_n)\Big|dy_1
\\
&=\FavP(E_n).
\end{align*}
Applying Theorem \ref{thm:comparison}, and then summing in $j$, we get
$$
|\mathcal{G}(E_n)|\lesssim \FavP(E_n)\lesssim \Fav(E_{\lfloor n/2\rfloor}).
$$
Furthermore, taking the limit $n\to\infty$ (or alternatively, applying the same argument as above to $E$), we get
$$
|\mathcal{G}(E)|=\FavP(E)=0.
$$

\end{proof}

As an application, we observe that Theorem \ref{thm:unioncircles} follows as an immediate consequence of Theorem \ref{thm-Gfunction}. 

\begin{proof}[Proof of Theorem \ref{thm:unioncircles}]
Let $0<a<b<\infty$ and $0\leq c<1$ be fixed real numbers. 
Let $G(x,y)=|x-y|^2 -r^2(x)$, where $r: \R^2\to [a,b]$ is a $C^2$ function. Then $\Gamma_x$ is a circle in the $y$-variables, centered at $x$ and of radius $r(x)$, and the set $\mathcal{G}$ in (\ref{eq:gammaunions}) is a union of such circles. We further have
$$
\nabla_y G=2(y-x),
$$
so that (\ref{Ggrady}) holds on any bounded subset $\mathcal{U}\subset\R^4$. The determinant in (\ref{jacobian-G}) is equal to
$$
8r\Big( (\nabla r)\cdot(x-y)-r\Big).
$$
If we assume that
\begin{equation}\label{grad-r-small}
|\nabla r|\leq c,
\end{equation}
then $|(\nabla r)\cdot(x-y)|\leq |\nabla r|\,|x-y|\leq cr$ on $\Gamma_x$. Hence we may choose the set $\mathcal{U}$ containing $\mathcal{G}$ so that (\ref{jacobian-G}) holds.
\end{proof}

With $E_n$ and $E$ defined as above, and assuming that (\ref{grad-r-small}) holds, 
Theorem \ref{thm-Gfunction} gives a bound on the size of a union of circles of radius $r(x)$ centered at points $x$ of $E_n$ and $E$. 
This is an important case that has been of interest in the literature. However, our
framework also covers more general cases, for instance ellipses with mildly varying parameters as shown below.

\begin{proof}[Application of Theorem \ref{thm-Gfunction}: Ellipses]
Let $0<c<C<\infty$. 
Let 
$$
G(x,y) = \frac{(x_1-y_1)^2}{a(x)^2} + \frac{(x_2-y_2)^2}{b(x)^2} -1,
$$
where $a,b: \R^2\to [c,C]$ are $C^2$ functions with all partial derivatives bounded strictly above by some constant $A\in[0,1/2)$ in absolute value. 

Then $\Gamma_x$ is an ellipse in the $y$-variables, centered at $x$ and of horizontal radius $a(x)$ and vertical radius $ b(x)$, and the set $\mathcal{G}$ in (\ref{eq:gammaunions}) is a union of such ellipses.

To apply Theorem \ref{thm-Gfunction}, it suffices to show that (\ref{Ggrady}) and (\ref{jacobian-G}) hold on the set $G(x,y)=0$; by continuity, this extends to some open neighbourhood of that set.
We have
$$
\nabla_y G= -2 \left( \frac{ (x_1-y_1)}{a(x)^2}, \,\, \frac{(x_2-y_2)}{b(x)^2}  \right),
$$
so that (\ref{Ggrady}) holds on any bounded subset $\mathcal{U}\subset\R^4$. 

Next, we consider the determinant in (\ref{jacobian-G}).
Compute 
$$G_1 = 2(x_1-y_1) \frac{ \left[ a(x) - (x_1-y_1) a_1(x)\right]  }{a(x)^3},$$

$$G_2 = 2(x_2-y_2) \frac{ \left[ b(x) - (x_2-y_2) b_2(x)\right]  }{b(x)^3},$$

$$G_3 =  -\frac{2(x_1-y_1)  }{a(x)^2} \,\,\, \text{ and } \,\,\,
G_4 =  -\frac{ 2(x_2-y_2)  }{b(x)^2} ,$$

 Further, $G_{32}=G_{41}=0$, and 
  $$G_{31} =  \frac{-2(a(x) - 2(x_1-y_1)a_1(x))  }{a(x)^3},$$
and
$$G_{42} =  \frac{-2(b(x)-2(x_2-y_2)b_2(x))  }{b(x)^3}.$$

Our bound on the partial derivatives of $a$ and $b$ implies that, on the set $G(x,y)=0$,
\begin{equation}\label{eq-ellipses1}
a- (x_1-y_1)a_1\geq a-A|x_1-y_1|\geq (1-A)a> 0,
\end{equation}
\begin{equation}\label{eq-ellipses2}
a- 2(x_1-y_1)a_1\geq a-2A|x_1-y_1|\geq (1-2A)a > 0,
\end{equation}
and similarly for the second coordinate. 

At least one of $|x_1-y_1|\geq a^2/2$ or $|x_2-y_2|\geq b^2/2$ holds.  
Assume the former. Then $|\nabla_y G|\geq |G_3|\gtrsim 1$.
We now check the Jacobian condition. 
The determinant in \ref{jacobian-G} equals 
$$
\det
\begin{pmatrix}
 0 & G_3 & G_4\\
 G_1& G_{31} & G_{41}\\
 G_2 & G_{32} & G_{42}
\end{pmatrix}
=
\det
\begin{pmatrix}
 0 & G_3 & G_4\\
 G_1& G_{31} & 0\\
 G_2 & 0 & G_{42}
\end{pmatrix}
= - [G_1G_3 G_{42} + G_2G_4G_{31}]. 
$$
It follows from (\ref{eq-ellipses1}), (\ref{eq-ellipses2}), and similar inequalities in the second variable, that
\begin{equation}\label{eq-ellipses3}
G_1 G_3\leq 0, \ \ G_2G_4\leq 0, \ \ G_{31}\leq 0,\ \ G_{42}\leq 0.
\end{equation}
Assuming again that $|x_1-y_1|\geq a^2/2$, we additionally have
$$
G_1G_3G_{42} \gtrsim 1,
$$
so that $G_1G_3 G_{42} + G_2G_4G_{31}\gtrsim 1$ as required. The case $|x_2-y_2|\geq b^2/2$ is similar, with $G_2G_4G_{31}\gtrsim 1$ instead.

With $E_n$ and $E$ defined as above,
Theorem \ref{thm-Gfunction} now yields an analogous conclusion to that of the circle application above for ellipses. 
\end{proof}

\section*{Acknowledgements}
I.{\L}. is supported in part by NSERC Discovery Grant GR010103.
A.M. is supported by an AMS-Simons Research Enhancement Grant for Primarily Undergraduate Institution Faculty.
K.T. is supported in part by the Simons Foundation Grant GR137264.

%

\bibliography{refs}
\bibliographystyle{abbrv}

\end{document}